\documentclass[12pt]{amsart}

\usepackage{amsfonts}
\usepackage{amssymb}
\usepackage{amsthm}
\usepackage{amsmath}
\usepackage{enumerate}
\usepackage{mathtools}
\usepackage{graphicx}
\usepackage{color}

\newtheorem{theorem}{Theorem}[section]
\newtheorem{lemma}[theorem]{Lemma}

\newtheorem{claim}[theorem]{Claim}
\newtheorem{proposition}[theorem]{Proposition}

\newtheorem*{theorem*}{Theorem}

\newtheorem{corollary}[theorem]{Corollary}
\newtheorem{definition}[theorem]{Definition}
\numberwithin{equation}{section}

\newcommand{\KK}{\mathcal{K}}

\newcommand{\MM}{\mathcal{M}}

\newcommand{\RR}{\mathbb{R}}

\newcommand{\eps}{\varepsilon}

\newcommand{\spt}[1]{\operatorname{spt}\left({#1}\right)}

\begin{document}
\title[Free boundary flow through cylindrical singularities]{Free boundary flow through cylindrical singularities}
\author{Yueheng Bao}\address{Department of Mathematics\\ University of Toronto\\Toronto, ON M5S 2E4\\ Canada}
\email{bao624@math.utoronto.ca}

\author{Robert Haslhofer}\address{Department of Mathematics\\ University of Toronto\\Toronto, ON M5S 2E4\\ Canada}
\email{roberth@math.toronto.edu}

\begin{abstract} We consider mean curvature flow with free boundary through cylindrical or half-cylindrical singularities, namely singularities of the types $\mathbb{R}^k\times S^{n-k}$, $\mathbb{R}^k_+\times S^{n-k}$ or $\mathbb{R}^k\times S^{n-k}_+$. Using the foundational results for free boundary Brakke flows by Edelen and the first author, and the recent classification of ancient asymptotically cylindrical flows by Bamler-Lai, we prove that all these singularities have a mean-convex neighborhood. Moreover, generalizing work of Hershkovits-White to the free boundary setting we show that the free boundary level set flow is nonfattening provided all singularities have a mean-convex neighborhood. We conclude that free boundary flow through singularities is well-posed as long as all singularities are of cylindrical or half-cylindrical type.

\end{abstract}
\date{\today}
\maketitle

\section{Introduction}

A family of hypersurfaces $M_t\subset\RR^{n+1}$ moves by mean curvature flow if the velocity at each point is given by the mean curvature vector. It is often of particular interest to consider hypersurfaces with free boundary, namely $M_t$ is now contained in some domain $\bar{\Omega}\subset\mathbb{R}^{n+1}$, satisfies the mean curvature flow equation in the interior $\Omega$, and meets $\partial \Omega$ perpendicularly. This free boundary flow has been studied quite extensively in the smooth setting, initiated by Stahl \cite{Stahl2,Stahl1} and Buckland \cite{Buckland}, and  pursued further by many others. However, to properly describe flows through singularities, which is most crucial for applications, one needs weak solutions. Suitable such notions are the free boundary level set flow, considered first by Giga and Sato \cite{GigaSato,Sato}, and the free boundary Brakke flow, introduced first by Mizuno-Tonegawa \cite{MizunoTonegawa}, and developed comprehensively by Edelen \cite{Edelen}.
Using this framework, Edelen, Ivaki, Zhu and the second author obtained a precise regularity and structure theory in the mean-convex case \cite{EHIZ}. Moreover, the second author also constructed a free boundary flow with surgery \cite{Haslhofer_fb_surgery}, and applied this to construct multiple free boundary minimal disks in convex domains in joint work with Ketover \cite{HaslhoferKetover} (see also the recent improvement \cite{SSWZ}).\\

However, the above results for free boundary flows correspond to what has been known already since 20 years for flows without boundary by the fundamental work of White \cite{White_size,White_nature} and Huisken-Sinestrari \cite{HuiskenSinestrari}, see also \cite{BrendleHuisken,HaslhoferKleiner,HaslhoferKleiner2}, so in a sense the theory of free boundary flows is still lacking far behind. Indeed, flows without boundary witnessed many exciting more recent breakthroughs, most notably the proof of the uniqueness conjecture for flows through mean-convex singularities by Hershkovits-White \cite{HershkovitsWhite}, the proof of the mean-convex neighborhood conjecture by Bamler-Lai \cite{BamlerLai2,BamlerLai1} (see also \cite{CHH,CHHW} for the earlier proof in the case of neck-singularities), the proof of the multiplicity-one conjecture by Bamler-Kleiner \cite{BamlerKleiner}, and the proof of the genericity conjecture by Chodosh-Choi-Mantoulidis-Schulze \cite{CCMS_generic1,CCS_generic2,CCMS_revisit}.\\

In the present paper, we prove free boundary versions of the former two conjectures (we plan to address the latter two elsewhere). Namely, we prove that the free boundary level set flow is nonfattening provided all singularities have a mean-convex neighborhood, and we prove that all cylindrical and half-cylindrical singularities indeed have such a mean-convex neighborhood, and in fact a canonical neighborhood. Crucial ingredients for this are the classification of ancient $k$-cylindrical flows by Bamler-Lai \cite{BamlerLai2,BamlerLai1}, as well as the recent foundational results for free boundary flows by the first author \cite{Bao}.
We conclude that free boundary flow through singularities is well-posed as long as all singularities are of cylindrical or half-cylindrical type.\\

\bigskip

\subsection*{Main results} 

Throughout this paper, we fix a smooth open convex domain $\Omega\subset\mathbb{R}^{n+1}$. Recall that a smooth hypersurface $M\subset\bar{\Omega}$ has free boundary if it meets $\partial \Omega$ orthogonally at its boundary $\partial M= M\cap \partial \Omega$. We will also often consider the two regions $K,K'\subset\bar{\Omega}$ enclosed by $M$, such that
\begin{equation}\label{encl_regions}
K\cup K'=\bar{\Omega}\quad\mathrm{and}\quad K\cap K'=M.
\end{equation}

To state our results let us recall some background from \cite{Bao}. Given any smooth compact free boundary hypersurface $M\subset\bar{\Omega}$, we can evolve it through singularities by the free boundary level set flow $F_t(M)$, the outer free boundary flow $M_t$, and the inner free boundary flow $M'_t$. As introduced first in \cite{GigaSato,Sato}, the free boundary level set flow of any closed set $X\subset\bar{\Omega}$ is simply the maximal family of closed sets $F_t(X)$ with $F_0(X)=X$ that satisfies the avoidance principle when compared with any smooth compact free boundary flow. To describe the other two flows, we start with the enclosed regions $K,K'$ as defined in \eqref{encl_regions} and consider the spacetime tracks of their free boundary level set flows,
\begin{equation}
\mathcal{K}=\bigcup_{t\geq 0} F_t(K)\times \{t\}\quad \mathrm{and}\quad \mathcal{K'}=\bigcup_{t\geq 0} F_t(K')\times \{t\}.
\end{equation} 
Working with the relative spacetime boundary $\partial \mathcal{X}=\mathcal{X}\setminus \mathrm{Int}_{\bar{\Omega}\times \mathbb{R}}(\mathcal{X})$, where $\mathrm{Int}_{\bar{\Omega}\times \mathbb{R}}(\mathcal{X})$ denotes the interior of $\mathcal{X}$ viewed as a subset of the topological space $\bar{\Omega}\times\mathbb{R}$, the outer and inner free boundary flow have then be defined in \cite{Bao}, generalizing the notions from \cite{HershkovitsWhite}, as
\begin{equation}
M_t= \{ x\in \bar{\Omega} \, | \, (x,t)\in \partial \mathcal{K}\}\quad \mathrm{and}\quad M_t'= \{ x\in \bar{\Omega} \, | \, (x,t)\in \partial \mathcal{K}'\}.
\end{equation} 
For convenience,  we will state the following definitions and results for the free boundary outer flow, but swapping the roles of $K$ and $K'$ one sees that they hold equally well for the free boundary inner flow.\\

If $x_0\in\Omega$ is an interior point, then similarly as in \cite{CHH} we say that $\mathcal{K}$ has a cylindrical singularity at $X_0=(x_0,t_0)$ if for some $\lambda_i\to \infty$ the parabolically rescaled flows $\mathcal{D}_{\lambda_i}(\mathcal{K}-X_0)$ converge locally smoothly with multiplicity-one up to rotation to a round shrinking solid cylinder
\begin{equation}
\left\{ \mathbb{R}^k \times \bar{B}^{n+1-k}(\sqrt{2(n-k)|t|})  \right\}_{t<0}.
\end{equation}
By convention this includes the case of spherical singularities for $k=0$.\\

If $x_0 \in \partial \Omega$ is a boundary point, then the rescaled domains $\lambda\cdot(\Omega-x_0)$ for $\lambda\to \infty$ converge up to rotation to the half-space $\mathbb{R}^{n+1}_{+}=\mathbb{R}^{n+1}\cap \{ x_1\geq 0\}$. Also note that $\mathbb{R}^{n+1}_{+}$ admits two types of free boundary half-cylinders with axis orthogonal or parallel to $\partial \mathbb{R}^{n+1}_{+}$, respectively. Motivated by this, we define half-cylindrical singularities as follows:

\begin{definition}[half-cylindrical singularities]\label{def_half_cyl} The outer free boundary flow has a half-cylindrical singularity at $X_0=(x_0,t_0)\in \partial\Omega\times \mathbb{R}_+$ if for some $\lambda_i\to \infty$ the parabolically rescaled flows $\mathcal{D}_{\lambda_i}(\mathcal{K}-X_0)$ converge locally smoothly with multiplicity-one up to rotation to either
\begin{equation}\label{sol_cyl1}
\left\{ \mathbb{R}^k_{+} \times \bar{B}^{n+1-k}(\sqrt{2(n-k)|t|}) \right\}_{t<0},
\end{equation}
or
\begin{equation}\label{sol_cyl2}
 \left\{  \bar{B}_{+}^{n+1-k}(\sqrt{2(n-k)|t|})\times\mathbb{R}^k \right\}_{t<0}.
\end{equation}
\end{definition}

Our first main theorem shows that all cylindrical and half-cylindrical singularities have a (two-sided) mean-convex neighborhood:

\begin{theorem}[mean-convex neighborhoods]\label{thm_mcn}If the outer free boundary flow $\{M_t\}_{t\geq 0}$ has a cylindrical or half-cylindrical singularity at $X_0=(x_0,t_0)\in \bar{\Omega}\times \mathbb{R}_+$, then there exists a constant $\delta=\delta(X_0)>0$, such that $K_t=F_t(K)$ for all $t_0-\delta^2< t_1 < t_2< t_0+\delta^2$ satisfies
\begin{equation}
K_{t_2}\cap B(x_0,\delta) \subseteq K_{t_1}\setminus M_{t_1}.
\end{equation}
\end{theorem}

This generalizes the resolution of Ilmanen's mean-convex neighborhood conjecture from \cite{BamlerLai2,CHH,CHHW} to the free boundary setting. In the course of the proof we also establish further interesting properties:

\begin{theorem}[canonical neighborhoods]\label{intro:cannghd}
If the outer free boundary flow $\{M_t\}_{t\geq 0}$ has a cylindrical or half-cylindrical singularity at $X_0=(x_0,t_0)\in \bar{\Omega}\times \mathbb{R}_+$ as defined above, then for every $\eps>0$, there exists a $
\delta=\delta(\varepsilon,X_0)>0$, such that for every $X\in \partial \mathcal{K}\cap P(X_0,\delta)$ we have:
\begin{enumerate}[(i)]
\item If $X$ is regular, then $H(X)\neq 0$ and the flow $\mathcal{D}_{|H(X)|^{-1}}(\partial\mathcal{K}-X)$ is $\eps$-close in $C^{\lfloor 1/\eps \rfloor}(P_{1/\eps}(0))$ to an ancient $k'$-cylindrical or half-cylindrical flow for some $k'\leq k$.\label{can_nbd_a}
\item If $X$ is singular, then every tangent flow at $X$ is a multiplicity-one shrinking cylinder or half-cylinder for some $k'\leq k$.\label{can_nbd_b}
\end{enumerate}
In particular, $(\partial \KK)^{\mathrm{sing}}\cap P(X_0,\delta)$ has dimension at most $k$.
\end{theorem}
Here, an ancient $k$-cylindrical flow is an ancient, unit-regular, integral Brakke flow in $\mathbb{R}^{n+1}$ whose tangent flow at $-\infty$ is a round shrinking $\mathbb{R}^{k}\times S^{n-k}$, and a half-cylindrical flow is the restriction of such a flow to a free boundary flow defined in a half-space. Any ancient $k$-cylindrical flow is up to splitting Euclidean factors either a sphere, or a translating bowl, or an ancient oval, or an oval-bowl by the recent classification by Bamler-Lai \cite{BamlerLai2,BamlerLai1} (see also \cite{CHH,CHHW} for the earlier classification for $k=1$, which in turn built on \cite{ADS1,ADS2,BrendleChoi1,BrendleChoi2}, and see \cite{CDDHS,ChoiHaslhofer,CHH_trans,CHH_wing,DH_no_rot,DH_shape} for the earlier classification of ancient noncollapsed flows in $\RR^4$).\\

Our other main result concerns uniqueness of free boundary flow through singularities, which is captured by nonfattening. To discuss this, similarly as in \cite{HershkovitsWhite},
we consider the fattening time
\begin{equation}
T_{\mathrm{fat}}= \inf \{ t > 0 \, | \,  F_t(M)\subset\bar{\Omega} \textrm{ has nonempty interior} \},
\end{equation}
and the discrepancy time
\begin{equation}
T_{\mathrm{disc}}= \inf \{ t > 0 \, | \, M_t,\, M_t' \textrm{ and } F_t(M) \textrm{ are not all equal} \}.
\end{equation}
Note that no discrepancy implies nonfattening, i.e. $T_{\mathrm{disc}}\leq T_{\mathrm{fat}}$. We can now state our theorem, which shows that free boundary flow through singularities is nonfattening (and in fact has no discrepancy) as long as all singularities have a mean-convex neighborhood:

\begin{theorem}[nonfattening for free boundary flows]\label{thm_nonfatt}
If $0<T\leq T_{\mathrm{disc}}$, and all singularities at time $T$ have a mean-convex neighborhood for the inner or outer free boundary flow, then $T<T_{\mathrm{disc}}$.
\end{theorem}

This generalizes the main result from \cite{HershkovitsWhite} to the free boundary setting. To be precise, the assumption of Theorem \ref{thm_nonfatt} says that for each singular point $X_0=(x_0,t_0)$, there is a $\delta=\delta(X_0)>0$, such that for all $t_0-\delta^2< t_1 < t_2\leq t_0$ we have either
\begin{equation}
K_{t_2}\cap B(x_0,\delta) \subseteq K_{t_1}\setminus M_{t_1}\quad\mathrm{ or }\quad K'_{t_2}\cap B(x_0,\delta) \subseteq K'_{t_1}\setminus M'_{t_1}.
\end{equation}
So this allows for both inwards and outwards shrinking (half)-cylinders, and only needs backwards/one-sided parabolic neighborhoods.\\

Combining Theorem \ref{thm_mcn} and Theorem \ref{thm_nonfatt} we immediately obtain:

\begin{corollary}[free boundary flow through cylindrical singularities]\label{cor_uniqueness}
Free boundary flow through singularities is well-posed as long as all singularities are of cylindrical or half-cylindrical type.
\end{corollary}

This generalizes the resolution of White's uniqueness conjecture for flows through cylindrical singularities from \cite{BamlerLai2,CHH,CHHW} to the free boundary setting. The conclusion of Corollary \ref{cor_uniqueness} is to some extent sharp, since conical singularities typically lead to nonuniqueness \cite{AIC,CDHS}.

\bigskip

\subsection*{Outline of the proofs}Given any half-cylindrical singularity of the outer free-boundary flow as in Definition \ref{def_half_cyl}, we first show that there exists a unit-regular, cyclic, free boundary integral Brakke flow, whose support is given by the outer flow,
and which has a multiplicity-one near the singularity. While the existence of some matching flow was already established in \cite{Bao}, here we need to argue more carefully to ensure that we actually get multiplicity-one. Another related issue is that for half-cylindrical singularities uniqueness of tangent flows is unknown. To get around this, using the isolation of cylinders in the space of shrinkers from \cite{CIM}, we show half-cylindrical tangent flows are unique, possibly up to rotation of their axis, which is enough for our purpose.\\

Having established these structural results about half-cylindrical singularities, we then generalize the proofs of the mean-convex and canonical neighborhood theorem from \cite{BamlerLai2,CHH,CHHW} to our setting. Here, we can also encounter blowup limits that are ancient unit-regular, cyclic, free boundary integral Brakke flows defined in a half-space. To deal with them we use the preservation under weak limits and reflections from \cite{Bao}, which enables us to apply the classification result from \cite{BamlerLai2}.\\

To prove Theorem \ref{thm_nonfatt}, we generalize the arguments from \cite{HershkovitsWhite} to our free boundary setting. Specifically, we first derive a general criterion, which says that the zero level of a function is the free boundary level set flow, provided the function is a free boundary subsolution in a neighborhood of the singular part, and satisfies several conditions in the regular part, including in particular an appropriate sign for the normal derivative at the boundary. We then construct a function with all the desired properties by interpolating between the arrival time function in the singular part and the signed distance function in the regular part. Here, most interestingly, regarding the boundary terms we observe that the normal derivative of the signed distance has the favourable sign thanks to convexity, and to facilitate the interpolation we construct a cutoff function with vanishing normal derivative.

\bigskip

{\bf Acknowledgements:} Y.B. has been supported by a Blyth Fellowship, a Mary H. Beatty Fellowship, a Department of Mathematics Graduate Program Award and an International Graduate Student Scholarship from the University of Toronto.
R.H. has been supported by the NSERC Discovery Grant RGPIN-2023-04419.\\

\section{Mean-convex neighborhoods for free boundary flows}

Let us recall some background relevant for this section. As introduced by Edelen \cite{Edelen}, a free boundary integral Brakke flow in $\bar{\Omega}$ is given by a family of Radon measures $\{\mu_t\}_{t\in I}$ supported in $\bar{\Omega}$, such that
\begin{equation}
\left. \int \phi(\cdot,t)\, d\mu_t\right|_{t=a}^b\leq \int_a^b \int \left( - |H_\ast |^2\phi+H_\ast \cdot D\phi + \partial_t\phi \right)\, d\mu_tdt
\end{equation}
for all nonnegative test functions $\phi$ satisfying the Neumann boundary condition $\nu_{\partial \Omega}\cdot D\phi=0$. Here, $\nu_{\partial \Omega}$ is the outwards unit normal of $\partial\Omega$,
\begin{equation}
H_\ast = H- 1_{\partial \Omega}(H\cdot  \nu_{\partial \Omega})\nu_{\partial\Omega},
\end{equation}
and at almost every time associated to $\mu_t$ there is an integral varifold $V_{\mu_t}$ with first variation
\begin{equation}
\delta V_{\mu_t}(X)=- \int H_\ast\cdot X\, d\mu_t\quad \textrm{for all $X$ tangential to $\partial\Omega$}.
\end{equation}
The support of $\mathcal{M}=\{\mu_t\}_{t\in I}$ is by definition the space-time set
\begin{equation}
\mathrm{spt}(\mathcal{M})=\overline{\bigcup_{t\in I} \mathrm{spt}(\mu_t)\times \{t\}}.
\end{equation}
By Edelen's monotonicity formula \cite[Theorem 5.5]{Edelen}, given any $X_0=(x_0,t_0)\in \mathrm{spt}(\mathcal{M})$, the reflected Gaussian area $\Theta_{\mathrm{refl}(\partial \Omega)}(\mathcal{M},X_0,r)$ is almost monotone, and gives rise to the reflected Gaussian density
\begin{equation}
\Theta_{\mathrm{refl}(\partial \Omega)}(\mathcal{M},X_0)= \lim_{r\to 0}\Theta_{\mathrm{refl}(\partial \Omega)}(\mathcal{M},X_0,r).
\end{equation}
All the free boundary integral Brakke flows that we encounter will be unit-regular and cyclic as in \cite[Definition 1.1]{Bao}. Namely, every $X_0\in \mathrm{spt}(\mathcal{M})$ with $\Theta_{\mathrm{refl}(\partial \Omega)}(\mathcal{M},X_0) =1$ is a regular point, and for almost every $t$ the associated $\mathbb{Z}_2$ flat chain $[V_{\mu_t}]$ satisfies $\mathrm{spt}(\partial [V_{\mu_t}])\subseteq \partial \Omega$.\\ 

\subsection{Structure of half-cylindrical singularities} In this subsection, we show that half-cylindrical singularities are unique and possess a suitable matching free boundary Brakke flow. We start with the latter:

\begin{proposition}[matching free boundary Brakke flow]\label{prop:matchbrakke}
Assume the outer free boundary flow has a half-cylindrical singularity at $X_0=(x_0,t_0)\in \partial\Omega\times \mathbb{R}_+$.
Then, there exists a unit-regular, cyclic free boundary integral Brakke flow $\mathcal{M}=\left\{\mu_t\right\}_{t \geq t_\ast}$, where $t_\ast=t_\ast(X_0)<t_0$, whose support is given by the outer flow $\{M_t\}_{t\geq t_\ast}$, such that one of the tangent flows to $\mathcal{M}$ at $X_0$ is a multiplicity-one half-cylinder, either $\{\mathbb{R}^k_+ \times S^{n-k}(0,\sqrt{2(n-k)|t|})\}_{t\leq 0}$ or $\{S_+^{n-k}(0,\sqrt{2(n-k)|t|})\times \mathbb{R}^k\}_{t\leq 0}$.
\end{proposition}

\begin{proof} By \cite[Theorem 1.8]{Bao}, there exists a unit-regular, cyclic free boundary integral Brakke flow $\mathcal{M}=\{\mu_t\}_{t\ge 0}$ starting at $\mu_0=\mathcal{H}^n\llcorner M_0$,
whose support is given by the outer flow $\{M_t\}_{t\geq 0}$, namely
\begin{equation}
\mathrm{spt}(\mathcal{M})=\bigcup_{t\geq 0} M_t.
\end{equation}
In particular, thanks to Edelen's monotonicity formula \cite[Theorem 5.5]{Edelen}, the Brakke flow $\mathcal{M}$ has bounded area ratios, namely
\begin{equation}\label{area_ratios_bdd}
\sup_{x\in\bar{\Omega}}\sup_{r\leq 1}\frac{\mu_t(B_r(x))}{r^n}\leq C.
\end{equation}

Now, according to Definition \ref{def_half_cyl} our assumptions says that there exists a sequence $\lambda_i\to\infty$, such that the parabolically rescaled outer flows
$\mathcal{K}^i:=\mathcal{D}_{\lambda_i}(\mathcal{K}-X_0)$
converge locally smoothly with multiplicity-one to a solid shrinking half-cylinder, given by either \eqref{sol_cyl1} or \eqref{sol_cyl2}.
Consider the parabolically rescaled Brakke flows
$\mathcal{M}^i := \mathcal{D}_{\lambda_i}(\mathcal{M}-X_0)$ along the same sequence of rescaling factors $\lambda_i\to \infty$.
By \cite[Proposition 6.2]{Edelen}, after passing to a subsequence, $\mathcal{M}^i$ converges to a free boundary integral tangent flow $\hat{\mathcal{M}}_{X_0}$ in $\mathbb{R}^{n+1}_+$, which can be doubled to obtain a self-similarly shrinking integral Brakke flow $\hat{\mathcal{M}}_{X_0}'$ in $\mathbb{R}^{n+1}$.
Since $\partial\mathcal{K}^i=\mathrm{spt}(\mathcal{M}^i)$ converges locally smoothly to a shrinking half-cylinder given by the boundary of either \eqref{sol_cyl1} or \eqref{sol_cyl2}, it follows that
\begin{equation}
\mathrm{spt}(\hat{\mathcal{M}}_{X_0}')=\bigcup_{t\leq 0}\left(\RR^k\times S^{n-k}(\sqrt{2(n-k)|t|}) \right).
\end{equation}
Moreover, by \cite[Theorem 6.4]{Edelen} the Gaussian density of the doubled tangent flow is equal to the reflected Gaussian density at $X_0$, namely
\begin{equation}
\Theta(\hat{\mathcal{M}}_{X_0}')=\Theta_{\mathrm{refl}(\partial \Omega)}(\mathcal{M},X_0).
\end{equation}
To address the potential issue of multiplicity, note that since shrinking cylinders have Gaussian area less than 2, and thus shrinking half-cylinders have Gaussian area less than 1, taking also into account \eqref{area_ratios_bdd} to control the Gaussian tail, we can find some $t_\ast < t_0$, such that
\begin{equation}
\frac{1}{\left(4 \pi\left(t_0-t_*\right)\right)^{n / 2}} \int_{\bar{\Omega}} \exp \left(\frac{-\left|x-x_0\right|^2}{4\left(t_0-t_*\right)}\right) d \mathcal{H}^n\llcorner M_{t_*}<1.
\end{equation}
Now we apply \cite[Corollary 5.5]{Bao} to get another unit-regular, cyclic free boundary integral Brakke flow $\MM=\{\mu_t\}_{t\ge t_*}$ with support given by $\{M_t\}_{t\geq t_\ast}$ and initial condition $\mu_{t_*}=\mathcal{H}^n\llcorner {M_{t_*}}$. Hence, taking into account again Edelen's monotonicity formula from \cite{Edelen}, provided $t_0-t_\ast$ is chosen small enough, for $r>0$ small enough we get
\begin{equation}
\Theta_{\mathrm{refl}(\partial \Omega)}(\mathcal{M},X_0,r)<2,
\end{equation}
and hence some tangent flow of $\MM$ at $X_0$ along a subsequence of $\{\lambda_i\}$ must be a shrinking half-cylinder with multiplicity-one as desired.
\end{proof}

To establish uniqueness of half-cylindrical tangent flows, we will use that cylinders are isolated in the space of shrinkers. Specifically, by \cite[Corollary 2.12]{CIM} there exists a constant $\kappa>0$ depending only on the dimension and a bound for the area ratios, such that if $\Sigma=\mathrm{spt}(\nu)$ is a varifold shrinker in $\mathbb{R}^{n+1}$ with area ratios bounded as in \eqref{area_ratios_bdd}, then
\begin{equation}\label{isolation_lemma}
d_{-1}\left(\mu_{\RR^k\times \sqrt{2(n-k)}S^{n-k}},\nu\right)< \kappa \quad \Rightarrow \quad \Sigma \textrm{ is a round cylinder}.
\end{equation} 
Here, we work with the distance between Radon measures defined by
\begin{equation}\label{equ:dv}
d_t(\mu,\nu):=\sum_{j\geq 1}2^{-j}\left| \int \varphi_j\Big(\tfrac{x}{\sqrt{|t|}}\Big) e^{-\frac{|x|^2}{4|t|}}\, d(\mu -\nu) \right| \quad (t<0),
\end{equation}
where $\{\varphi_j\}_{j\geq 1}$ is a countable dense subset of the unit ball in $C_c(\RR^{n+1})$, which for later use we choose such that
\begin{equation}\label{choice_basis}
\varphi_{2j+1}(x_1,x_2,\ldots,x_{n+1})=\varphi_{2j}(-x_1,x_2,\ldots,x_{n+1}).
\end{equation}

Now, given a free boundary integral Brakke flow $\mathcal{M}=\{\mu_t\}$ in $\bar{\Omega}$, for any point $X_0=(x_0,t_0)\in\partial \Omega\times\mathbb{R}_+$ and scale $r>0$, we consider the parabolically rescaled Brakke flow $\mathcal{D}_{1/r}(\mathcal{M}-X_0)=\{\mu_t^{X_0,r}\}$. This is a free boundary integral Brakke flow in $\tfrac{1}{r}\cdot(\Omega-x_0)$, but in particular constitutes a family of Radon measures in $\mathbb{R}^{n+1}$, so it makes sense to consider its distance to a free boundary Brakke flow defined in $\RR^{n+1}_+$.

\begin{lemma}[quantitative rigidity]\label{lem:quanrigid}
For every $\varepsilon>0$, there exists $\delta=\delta(X_0)>0$ and $r_0=r_0(X_0)>0$, such that if $X_0\in\mathrm{spt}(\mathcal{M})$ and
\begin{equation}\label{density_drop}
\Theta_{\mathrm{refl}(\partial \Omega)}\left(\mathcal{M}, X_0, 2r\right)-\Theta_{\mathrm{refl}(\partial \Omega)}\left(\mathcal{M}, X_0,r\right) < \delta
\end{equation}
for some $0<r<r_0$, then there exists a self-similarly shrinking free boundary integral Brakke flow $\mathcal{N}=\{\nu_t\}_{t\leq 0}$ in  $\RR^{n+1}_+$, such that
\begin{equation}\label{close_to_shrinker}
d_t(\mu^{X_0,r}_t,\nu_t)< \eps\qquad \forall t\in[-4,-1].
\end{equation}
\end{lemma}

\begin{proof}
We argue similarly as in \cite[Lemma 3.2]{CHN} and \cite[Proposition 2.13]{CIM}. 
Suppose towards a contradiction there is some $\eps>0$, such that \eqref{density_drop} holds with $r_i=\delta_i=1/i$, 
but for every self-similarly shrinking free boundary integral Brakke flow $\mathcal{N}=\{\nu_t\}_{t\leq 0}$ in  $\RR^{n+1}_+$ we have
\begin{equation}\label{eq_dist_cont}
d_{t_i}(\mu^{X_0,r_i}_{t_i},\nu_{t_i})\geq \eps\qquad \textrm{for some } t_i\in[-4,-1].
\end{equation}
However, by \cite[Theorem 6.4]{Edelen} a subsequence of $\mathcal{M}^i=\{ \mu^{X_0,r_i}_{t}\}$ converges to a self-similarly shrinking free boundary integral Brakke flow in $\RR^{n+1}_+$. This yields a contradiction with \eqref{eq_dist_cont}, and thus proves the lemma.
\end{proof}

\begin{proposition}[uniqueness of half-cylindrical tangent flows]\label{prop:unitan}If $\mathcal{M}$ is a free boundary integral Brakke flow, and some tangent flow at $X_0\in\partial\Omega\times\mathbb{R}_+$ is a multiplicity-one shrinking half-cylinder, then they all are.
\end{proposition}

\begin{proof}
We will adapt the proof of \cite[Theorem 0.2]{CIM} to our setting. Let $\delta>0$ and $r_0>0$ be the constants from Lemma \ref{lem:quanrigid} applied with $\eps=\kappa/8$. By assumption, after decreasing $r_0$ we can arrange that $\mu_t^{X_0,r_0}$  satisfies \eqref{close_to_shrinker} with $\{\nu^0_t\}_{t\leq 0}$ a multiplicity-one shrinking half-cylinder and that \eqref{density_drop} holds for all $r\leq r_0$. Considering the dyadic scales $r_j=2^{-j}r_0$, the lemma then gives us self-similarly shrinking free boundary integral Brakke flows $\{\nu^j_t\}_{t\leq 0}$ in  $\RR^{n+1}_+$, such that 
\begin{equation}\label{close_to_shrinker_again}
d_t(\mu^{X_0,r_j}_t,\nu^j_t)< \kappa/8\qquad \forall t\in[-4,-1].
\end{equation}
In particular, by the triangle inequality this yields $d_{-1}(\nu^j_{-1},\nu^{j+1}_{-1})< \kappa/4$ for all $j$. After doubling (c.f. \cite[Proposition 3.1]{Edelen}), since we have arrange that \eqref{choice_basis} holds, we still get
\begin{equation}
d_{-1}\left(({\nu}^{j}_{-1})',({\nu}^{j+1}_{-1})'\right)< \kappa\qquad \forall j\geq 0,
\end{equation}
so \eqref{isolation_lemma} yields that all the $\{\nu^j_t\}_{t\leq 0}$ are multiplicity-one shrinking half-cylinders. In light of \eqref{isolation_lemma} and \eqref{close_to_shrinker_again}, this implies the assertion.
\end{proof}

\begin{corollary}[uniqueness of half-cylindrical singularities]\label{cor_uni_sing}
If the outer free boundary flow $\{M_t\}_{t\geq 0}$ has a half-cylindrical singularity at $X_0$ (see Definition \ref{def_half_cyl}), then there exists a matching unit-regular, cyclic free boundary integral Brakke flow $\mathcal{M}=\{\mu_t\}_{t\geq t_\ast}$, such that all tangent flows at $X_0$ are multiplicity-one half-cylinders. In particular, $\mathcal{D}_{\lambda_i}(\mathcal{K}-X_0)$ subconverges to a solid shrinking-half cylinder for all $\lambda_i\to \infty$.
\end{corollary}

\begin{proof}
This follows combining Proposition \ref{prop:matchbrakke} and Proposition \ref{prop:unitan}.
\end{proof}

\bigskip

\subsection{Proof of the mean-convex neighborhood theorem}

The crucial ingredient is the classification of ancient $k$-cylindrical flows by Bamler-Lai \cite{BamlerLai2} (see also the earlier work \cite{CHH,CHHW} for the case $k=1$). Let us extract the relevant consequences for our purpose:

\begin{lemma}[ancient $k$-cylindrical flows]\label{std_lemma} The following holds:
\begin{enumerate}[(i)]
\item Ancient $k$-cylindrical flows are convex, $1$-noncollapsed, and smooth at all times prior to possible extinction.  In particular, their regularity scale $R$ and their mean curvature scale $H^{-1}$ are comparable.\label{std_lemma_a}
\item Every nontrivial blowup limit of an ancient $k$-cylindrical flow is an ancient $k'$-cylindrical flow for some $k'\leq k$.\label{std_lemma_b}
\end{enumerate}
\end{lemma}

\begin{proof}
Assertion \eqref{std_lemma_a} is an immediate corollary of the classification by Bamler-Lai \cite{BamlerLai2}. To prove \eqref{std_lemma_b}, note that any blowup limit is also noncollapsed and asymptotic to a cylinder by \cite[Theorem 1.14]{HaslhoferKleiner}. Since the entropy cannot increase in the limit, we see that $k'\leq k$.
\end{proof}

We first establish the canonical neighborhood property (Theorem \ref{intro:cannghd}), which we restate here in the new case of boundary singularities:

\begin{theorem*}[canonical neighborhoods]\label{prop:cannghd}
If the outer free boundary flow has a half-cylindrical singularity at $X_0=(x_0,t_0)\in \partial\Omega\times \mathbb{R}_+$ as in Definition \ref{def_half_cyl}, then for every $\eps>0$, there exists a $
\delta=\delta(\varepsilon,X_0)>0$ such that for every $X\in \partial \mathcal{K}\cap P(X_0,\delta)$ we have:
\begin{enumerate}[(i)]
\item If $X$ is regular, then $H(X)\neq 0$ and the flow $\mathcal{D}_{|H(X)|^{-1}}(\partial\mathcal{K}-X)$ is $\eps$-close in $C^{\lfloor 1/\eps \rfloor}(P_{1/\eps}(0))$ to an ancient $k'$-cylindrical or half-cylindrical flow for some $k'\leq k$.\label{can_nbd_a}
\item If $X$ is singular, then every tangent flow at $X$ is a multiplicity-one shrinking cylinder or half-cylinder for some $k'\leq k$.\label{can_nbd_b}
\end{enumerate}
In particular, $(\partial \mathcal{K})^{\mathrm{sing}}\cap P(X_0,\delta)$ has dimension at most $k$.
\end{theorem*}

\begin{proof}By Corollary \ref{cor_uni_sing} there exists a unit-regular, cyclic, free boundary integral Brakke flow $\mathcal{M}=\{\mu_t\}_{t\geq t_\ast}$ with $\mathrm{spt}(\mathcal{M})=\partial\mathcal{K}\cap\{t\geq t_\ast\}$, and every tangent flow of $\mathcal{M}$ at $X_0$ is a multiplicity-one shrinking half-cylinder given by the boundary of either \eqref{sol_cyl1} or \eqref{sol_cyl2}.\\

Suppose towards a contradiction that \eqref{can_nbd_a} fails for some regular points $X_i \rightarrow X_0$. Consider the regularity scale $R(X_i)$, namely the largest $r>0$, such that $\sup_{\partial\mathcal{K}\cap P(X_i,r)}|A|\leq 1/r$. Observe that $R(X_i)\to 0$ and consider the blowup sequence  $\mathcal{M}^{i}=\mathcal{D}_{R(X_i)^{-1}}(\mathcal{M}-X_i)$.
By Edelen's compactness theorem \cite[Theorem 4.14]{Edelen} a subsequence converges to an ancient integral Brakke flow $\mathcal{M}^\infty$, which is either a flow without boundary defined in $\mathbb{R}^{n+1}$ or a flow with free boundary defined in $\mathbb{R}^{n+1}_+$ that can be doubled (see \cite[Proposition 4.6]{Edelen}) to obtain a flow $(\mathcal{M}^\infty)'$ without boundary in $\mathbb{R}^{n+1}$.
Let $\hat{\mathcal{M}}$ be $\mathcal{M}^\infty$ in the first case and $(\mathcal{M}^\infty)'$ in the second case. The flow $\hat{\mathcal{M}}$ is an ancient unit-regular, cyclic, integral Brakke flow in $\mathbb{R}^{n+1}$ thanks to the preservation under weak limits and reflections from \cite[Theorem 1.2 and Theorem 1.4]{Bao}.
Moreover, since $X_i\to X_0$ and every tangent flow at $X_0$ is a multiplicity-one shrinking half-cylinder, remembering also \eqref{isolation_lemma}, we infer that $\hat{\mathcal{M}}$ is either an ancient asymptotically cylindrical flow or a blowup thereof. Together with Lemma \ref{std_lemma} we conclude that  \eqref{can_nbd_a} actually holds for $X_i$ for $i$ large enough, which gives the desired contradiction.\\

Next, suppose towards a contradiction there are singular points $X_i \to X_0$ with a non-(half)cylindrical tangent flow. Then, by isolation of the cylinder in the space of shrinkers from \eqref{isolation_lemma}, we can find $\lambda_i \to \infty$, such that $\mathcal{M}^i=\mathcal{D}_{\lambda_i}(\MM - X_i)$ is not $\eps$-close to a cylinder or half-cylinder at any scale between $1/i$ and $i$. However, arguing as above we see that $\MM^i$ converges to a blowup of an ancient asymptotically cylindrical flow, which gives a contradiction for large $i$, and thus proves \eqref{can_nbd_b}.\\

Finally, we have to estimate the Hausdorff dimension of the singular set with respect to the parabolic metric
\begin{equation}
d((x,t),(x^{\prime},t^{\prime}))=\max\{|x-x^{\prime}|,|t-t^{\prime}|^{\frac12}\}.
\end{equation}
If the $D$-dimensional Hausdorff measure $\mathcal{H}^D_d$ of $(\partial \mathcal{K})^{\mathrm{sing}}\cap P(X_0,\delta)$ is positive, then considering the $D$-dimensional Hausdorff content
\begin{equation}
\mathcal{H}^{D,\infty}_d(\mathcal{X})=\inf\left\{ \sum_{i=1}^\infty \sup_{X,X'\in\mathcal{U}_i} d(X,X')^D \, : \, \mathcal{X}\subset \bigcup_{i=1}^\infty \mathcal{U}_i \right\},
\end{equation}
we can find $X\in (\partial \mathcal{K})^{\mathrm{sing}}\cap P(X_0,\delta)$ and $r_i\to 0$, such that
\begin{equation}
\lim_{i\to \infty} \frac{ \mathcal{H}^{D,\infty}_d((\partial \mathcal{K})^{\mathrm{sing}}\cap P(X,r_i))}{r_i^D}>0,
\end{equation}
see e.g. \cite[Section 2.3]{EvansGariepy}. Hence, we can find a tangent flow at $X$, whose singular set has parabolic Hausdorff dimension at least $D$ (here we used that limits of singular points are singular thanks to the $\eps$-regularity theorem from \cite{Edelen}). By our classification of tangent flows from item \eqref{can_nbd_b} we must have $D\leq k$, as desired.
\end{proof}

As a consequence of the canonical neighborhood theorem and its proof we can now establish the mean-convex neighborhood (Theorem \ref{thm_mcn}), which we restate here in the following equivalent form:

\begin{theorem*}[mean-convex neighborhoods] If the outer free boundary flow has a cylindrical or half-cylindrical singularity at $X_0=(x_0,t_0)\in \bar{\Omega}\times \mathbb{R}_+$, then there exists a constant $\delta=\delta(X_0)>0$, such that the outer free boundary flow is strictly mean-convex in $P(X_0,\delta)$.
\end{theorem*}

\begin{proof} We start by showing that each point is visited at most once:
\begin{claim}[arrival]\label{claim_arrivial}
There exists a $\delta_0>0$, such that for every $x\in {B}(x_0,\delta_0)$ we have $(x,t)\in\partial \KK$ for at most one $t\in (t_0-\delta_0^2,t_0+\delta_0^2)$. \label{claim_arrival}
\end{claim}

\begin{proof}We argue similarly as in \cite[Claim 7.2]{BamlerLai2}.
Suppose towards a contradiction that there are $x_i\to x_0$, such that $(x_i,t_i),(x_i,t_i')\in\partial \KK$ for two distinct times $t_i'<t_i$ converging to $t_0$. 
Modifying what we did in the above proof of the canonical neighborhood theorem, we now take
\begin{equation}\label{choice_of_ri}
r_i=\max\left(R(x_i,t_i),\sqrt{t_i-t_i^{\prime}}\right).
\end{equation}
After passing to a subsequence we can assume that the flows $\mathcal{M}^i=\mathcal{D}_{1/r_i}\left(\mathcal{M}-\left(x_i, t_i\right)\right)$ converge to a limit $\mathcal{M}^\infty$ and that $\hat{t}_i=r_i^{-2}\left(t_i^{\prime}-t_i\right) \rightarrow t_\infty \in[-1,0]$.
By construction we have $(0, 0),\left(0, t_\infty\right) \in \operatorname{spt} \mathcal{M}^\infty$. 
If $t_\infty \neq 0$, then we can directly obtain a contradiction by taking the double (if necessary) of $\mathcal{M}^{\infty}$ and applying Lemma \ref{std_lemma}.
If $t_{\infty}=0$, then considering \eqref{choice_of_ri} we see that $r_i=R(x_i,t_i)$ for $i$ large. In particular, we have smooth convergence $\mathcal{M}^i \rightarrow \mathcal{M}^{\infty}$ near $( 0, 0 )$,  and hence for large $i$ the flows $\mathcal{M}^i$ have positive mean curvature near $(0,0)$, which yields the desired contradiction with $(0, 0),\left(0, \hat{t}_i\right) \in \operatorname{spt} \mathcal{M}^i$.
\end{proof}
Finishing the proof of the theorem, we get a well-defined arrival time\footnote{As an aside, which is not needed in this section, we can observe that the function $(x,t)\mapsto u(x)-t$ solves the free boundary level set equation.}
\begin{equation}
u: B(x_0,\delta) \cap \bar{\Omega}\to (t_0-\delta_0^2,t_0+\delta_0^2),
\end{equation}
which is a continuous function that satisfies $u(x)=t$ if and only if $(x,t)\in\partial \KK\cap (t_0-\delta_0^2,t_0+\delta_0^2)$. At first, by Claim \ref{claim_arrivial}  this function is only defined on a subset $D\subseteq  B(x_0,\delta) \cap \bar{\Omega}$, but choosing $\delta\ll \delta_0$ small enough, Theorem \ref{intro:cannghd} implies that actually $D= B(x_0,\delta) \cap \bar{\Omega}$. Having established the arrival time of the outer flow, and remembering that the tangent flow at $X_0$ moves inwards, we can conclude that
\begin{equation}
K_{t_2}\cap B(x_0,\delta)\subseteq K_{t_1}\setminus M_{t_1}
\end{equation}
whenever $t_0-\delta^2< t_1 < t_2 <t_0+\delta^2$. This proves the theorem.
\end{proof}

\bigskip

\section{Uniqueness of free boundary flow through singularities}

Let us recall the definition of free boundary subsolutions \cite{GigaSato,Bao}, which will be used throughout this section. Given any open set $X\subseteq\RR^{n+1}$, a family of closed subsets of $X\cap \bar{\Omega}$ is a free boundary subsolution in $X$ if it satisfies the avoidance principle with respect to any smooth compact mean curvature flow in $X\cap \bar{\Omega}$ with free boundary on $X\cap\partial \Omega$.\\

\subsection{A characterization of free boundary level set solutions}
In this subsection, we provide a criterion to detect solutions of the free boundary level set flow. We start with the following basic lemma:

\begin{lemma}[free boundary subsolutions]\label{basic_lemma}
Let $X\subset \mathbb{R}^{n+1}$ be an open set, and suppose $w: X\cap\bar{\Omega} \times[0, T] \rightarrow \mathbb{R}$ is a continuous function.
\begin{enumerate}[(i)]
\item If for every $c\geq 0$ the family $\{w(\cdot,t)\geq c\}_{t\in [0,T]}$ is a free boundary subsolution in $X$, then for every $\alpha\geq 0$ and $c\geq 0$ the family $\{w(\cdot,t)\geq ce^{\alpha t}\}_{t\in[0,T]}$ is also a free boundary subsolution in $X$.\label{basic_lemma_a}
\item\label{basic_lemma_b} If $w$ is smooth with $Dw\neq 0$, and $c\in \mathbb{R}$, then $\{ w(\cdot,t)\geq c\}_{t\in [0,T]}$ is a free boundary subsolution in $X$ if and only if
\begin{equation*}
\left\{
\begin{array}{ll}
w_t- |Dw|\mathrm{div}\left(\frac{Dw}{|Dw|}\right)\leq 0 & \text{on } \{w=c\}, \\
 Dw \cdot \nu_{\partial \Omega}\leq 0 & \text{on } \{w=c\}\cap \partial\Omega.
\end{array}
\right.
\end{equation*}
\end{enumerate}
\end{lemma}

\begin{proof}Assertion \eqref{basic_lemma_a} follows directly from the definition of free boundary subsolutions observing that if $\{w(\cdot,t)\geq ce^{\alpha t}\}_{t\in[0,T]}$ violates the avoidance principle at some $\hat{t}\leq T$, then so does $\{w(\cdot,t)\geq \hat{c}\}_{t\in [0,T]}$, where $\hat{c}=c\exp(\alpha \hat{t})$.\\
In part \eqref{basic_lemma_b} the hypersurface $M_t=\{ w(\cdot,t)= c\}$ is smooth and its outwards unit normal and outwards speed are given by $\nu=-Dw/|Dw|$ and $s=w_t/|Dw|$.  Since by the classical maximum principle we have avoidance if and only if $s\leq -\textrm{div}(\nu)$ on $M_t$ and $\nu\cdot \nu_{\partial \Omega}\geq 0$ on $M_t\cap \partial\Omega$, this concludes the proof of the lemma.
\end{proof}

We can now detect solutions via the following characterization:

\begin{theorem}[characterization of level set solutions]\label{thm:hwcpam15}
Let $Y,Z\subset\mathbb{R}^{n+1}$ be bounded open sets. Suppose $\{ M_t\}_{t\in [0,T]}$ is a free boundary subsolution of compact sets contained in $(Y\cup Z)\cap\bar{\Omega}$, and suppose that there is a continuous function $w: (\overline{Y\cup Z})\cap\bar{\Omega}\times [0,T] \rightarrow \mathbb{R}$, such that:
\begin{enumerate}[(i)]
\item\label{assumpt_i} $w(x, t)=0$ if and only if $x \in M_t$.
\item\label{assumpt_ii} For any $c \ge 0$, 
$\{w(\cdot, t)\ge c\}\cap Y$
and
$\{ w(\cdot, t)\le -c\}\cap Y$
define free boundary subsolutions in $Y$.  
\item\label{assumpt_iii} On $\bar{Z}\cap\bar{\Omega}$ the function $w$ is smooth with $Dw\neq 0$, and on
 $Z\cap\partial \Omega$ we have $D w \cdot \nu_{\partial \Omega}\leq 0$ when $w\geq 0$ and $D w \cdot \nu_{\partial \Omega}\geq 0$ when $w\leq 0$. 
\end{enumerate}
Then, $\{M_t\}_{t\in[0,T]}$ is the free boundary level set flow of $M_0$ in $\bar{\Omega}$.
\end{theorem} 

\begin{proof}
We will generalize the argument from \cite[Theorem 3.9]{HershkovitsWhite} to our setting. Let us abbreviate
\begin{equation}
\Phi[w]=w_t-|D w| \textrm{div}\left(\frac{D w}{|D w|}\right),\qquad \phi[w]=Dw \cdot \nu_{\partial \Omega}.
\end{equation}
Applying Lemma \ref{basic_lemma} for $\pm w$ yields $\Phi[w]=0$ on $\{w=0\}\cap Z$. Thus, by compactness there is some $\alpha<\infty$, such that $|\Phi[w]|\leq \alpha|w|$ on $Z\cap\bar{\Omega}$. Consider the function $\tilde{w}=e^{-\alpha t} w$, and observe that
\begin{equation}
\Phi[ \tilde{w}] \leq 0 \textrm{ on } \{\tilde{w} \geq 0\}\cap Z,\quad \phi[ \tilde{w}] \leq 0 \textrm{ on } \{\tilde{w} \geq 0\}\cap Z\cap \partial\Omega,
\end{equation}
where we also used assumption \eqref{assumpt_iii}.
Thus, again by Lemma \ref{basic_lemma}, for any $c\geq 0$ the family $\{\tilde{w}(\cdot,t)\geq c\}_{t\in [0,T]}$ is a free boundary subsolution in $Z$. Moreover, assumption \eqref{assumpt_ii} and the lemma tell us that this family is also a free boundary subsolution in $Y$. It follows that $\{\tilde{w}(\cdot,t)\geq c\}_{t\in [0,T]}$ is a free boundary subsolution in $Y\cup Z$. The same holds for $-\tilde{w}$. Since the union of subsolutions is a subsolution, we thus infer that
\begin{equation}
\{|\tilde{w}(\cdot,t)|\geq c\}_{t\in [0,T]} \textrm{ is a free boundary subsolution in } Y\cup Z.
\end{equation}
Now, let $\eta$ be the minimum value of $|\tilde{w}|$ on $\partial(\overline{Y \cup Z})\cap\bar{\Omega} \times[0, T]$, and observe that $\eta>0$ thanks to assumption \eqref{assumpt_i}.
It follows that for $0<c<\eta$ the family
\begin{equation}\label{eq_complement}
\bar{\Omega}\setminus\{|\tilde{w}(\cdot, t)|<c\}= \{|\tilde{w}(\cdot, t)|\geq c\}\cup \bar{\Omega}\setminus(\overline{Y\cup Z})
\end{equation}
is a free boundary subsolution in $\bar{\Omega}$. 
Since the free boundary level set flow $F_t(M_0)$ and the set from \eqref{eq_complement} are disjoint at time $0$, they remain disjoint. As this is true for all small $c>0$, this shows that
\begin{equation}
F_t(M_0) \subseteq\{w(\cdot, t)=0\}=M_t\qquad \forall t\in[0,T].
\end{equation}
Since $F_t(M_0)$ is the maximal free boundary subsolution in $\bar{\Omega}$ starting at $M_0$, we conclude that $M_t=F_t(M_0)$ for all $t\in [0,T]$, as desired.
\end{proof}

\bigskip

\subsection{Proof of the nonfattening theorem}

In this subsection, we prove the nonfattening theorem for free boundary flows with mean-convex singularities. We start with the following technical lemma:

\begin{lemma}[cutoff function]\label{lem:fbcut}
Given any open sets $U\Subset U'\Subset \RR^{n+1}$, there exists a smooth function $\phi: \mathbb{R}^{n+1} \rightarrow[0,1]$, such that $\phi|_{\bar{U}} \equiv 1$, $\spt\phi\subset U'$, and $D \phi\cdot \nu_{\partial \Omega}=0$ on the boundary $\partial \Omega$. 
\end{lemma}

\begin{proof}Choose open sets $V$ and $V'$, such that $U\Subset V\Subset V'\Subset U'$.
By Urysohn's lemma, there exists a smooth function $\psi: \mathbb{R}^{n+1} \rightarrow[0,1]$, such that $\psi|_{\bar{V}} \equiv 1$ and $\spt\psi\subset V'$.
To modify our function near $\partial\Omega$ to arrange the free boundary condition, we consider the signed distance
\begin{equation}
d(x)=\left\{
\begin{array}{rl} 
d(x,\partial \Omega) & x\in \Omega,\\
-d(x,\partial \Omega) & x\in \mathbb{R}^{n+1}\setminus \Omega.
\end{array}
\right.
\end{equation}
and the nearest point projection $\pi(x)\in \partial \Omega$, which are both smooth in a neighborhood of $\partial \Omega$. Given $\eps>0$, let $\eta:\RR\to [0,1]$ be a smooth function, such that $\eta|_{[-\eps,\eps]}\equiv 1$ and $\mathrm{spt}(\eta) \subset [-2\eps,2\eps]$, and set
\begin{equation}
\phi(x) = \eta(d(x))\,\psi(\pi(x)) + (1-\eta(d(x)))\,\psi(x).
\end{equation}
For $\eps$ small enough this function has all the desired properties.
\end{proof}

We will now prove Theorem \ref{thm_nonfatt}, which we restate in the following equivalent form:

\begin{theorem*}[nonfattening for free boundary flows]
If $0<T\leq T_{\mathrm{disc}}$, and all singularities of the outer free boundary flow at time $T$ have a mean-convex or mean-concave neighborhood, then $T<T_{\mathrm{disc}}$.
\end{theorem*}

\begin{proof}
We will adapt the proof of \cite[Theorem 3.1]{HershkovitsWhite} to our setting. Denote by $S_t$ the set of all singular points of the outer flow at time $t$.  By assumption all $x\in S_T$ have a mean-convex or mean-concave neighborhood of some positive radius $\delta_x>0$. By compactness, we can cover $S_T$ by finitely many balls $B(x_i,\delta_{x_i}/5)$ with $x_i\in S_T$.
Choose an open subset $W\subset\RR^{n+1}$ with smooth boundary, with
\begin{equation}
\bigcup_i B\left(x_i, \delta_{x_i}/ 4\right)\subseteq W \subseteq \bigcup_i B\left(x_i, \delta_{x_i} / 3\right),
\end{equation}
such that $\partial W$ is transversal to $M_T$ and $\partial \Omega$. Let $P$ and $N$ be the union of components of $W$ that contain points $x_i$ of mean-convex type and mean-concave type, respectively. Note that $P$ and $N$ are disjoint and that setting $\delta=\min_i \delta_{x_i}$ for $T-\delta^2 < t_1<t_2 \leq T$ we have
\begin{equation}\label{inclusions_backw}
K_{t_2} \cap \overline{P} \subseteq K_{t_1}\setminus M_{t_1}\quad\mathrm{and}\quad K_{t_1} \cap \overline{N} \subseteq K_{t_2}\setminus M_{t_2}.
\end{equation}
Since $P$ and $N$ are open, and the space-time singular set is closed, we can find $\eps<\delta^2$, such that
\begin{equation}\label{sing_cont_in_pn}
S_{T,\eps}:=\bigcup_{t\in [T-\eps,T+\eps]} S_t \subset P \cup N.
\end{equation}
Moreover, by decreasing $\eps$ we can ensure that $M_t$ intersects $\partial(P\cup N)$ transversely with nonvanishing mean curvature for all $t\in [T-\eps,T+\eps]$. In particular, this yields that for $T-\eps \leq t_1<t_2 \leq T+\eps$ we have
\begin{equation}\label{boundary_incl}
K_{t_2} \cap \partial P \subseteq K_{t_1}\setminus M_{t_1}\quad\mathrm{and}\quad K_{t_1} \cap \partial N \subseteq K_{t_2}\setminus M_{t_2}.
\end{equation}
Let us show that this implies that \eqref{inclusions_backw} also holds forwards in time:\footnote{This step can be omitted by assuming a priori that the mean-convex neighborhoods are two-sided, as provided by the mean-convex neighborhood theorem.}

\begin{claim}[inclusion]
For $T-\eps \leq t_1<t_2 \leq T+\eps$ we have
\begin{equation}\label{inclusion}
K_{t_2} \cap \overline{P} \subseteq K_{t_1}\setminus M_{t_1}\quad\mathrm{and}\quad K_{t_1} \cap \overline{N} \subseteq K_{t_2}\setminus M_{t_2}.
\end{equation}
\end{claim}

\begin{proof}Recall that  $\{M_t\}$ is a free boundary subsolution in $\bar{\Omega}$ by \cite[Proposition 5.2]{Bao}, and that $K_t=F_t(K)$ is a solution by definition. In particular, both are free boundary subsolutions in $P$ and $N$. Now, in light of \eqref{inclusions_backw}, it suffices to suppose towards a contradiction that there is a first time $t_2 \in (T,T+\eps)$, where \eqref{inclusion} fails for some $t_1\in[T,t_2)$. Then, by \eqref{boundary_incl}  we infer that $\{K_{t\pm(t_2-t_1)}\}_{t\geq T-\eps}$ and $\{M_t\}_{t\geq T-\eps}$ touch each other at a nonempty compact subset of $P$ or $N$ at time $t=t_1$ or $t=t_2$, respectively. This contradicts the fact that free boundary subsolutions avoid each other by \cite{GigaSato}.
\end{proof}

Continuing the proof of the theorem, let
\begin{equation}
X=\bigcup_{t \in(T-\eps,T+ \eps)} M_t \cap(P\cup N)\cap\bar{\Omega},
\end{equation}
and consider the arrival time $u: X \rightarrow(T-\eps,T+ \eps)$ defined by
\begin{equation}
u(x)=\textrm {the unique $t\in(T-\eps,T+\eps)$ such that }  x\in M_t.
\end{equation}
Moreover, we define a function $f:X\times \RR_{+}\to \RR$ by
\begin{equation}\label{equ:defoff}
f(x, t)= \begin{cases}u(x)-t & \text { if } x \in P, \\ t-u(x) & \text { if } x \in N.\end{cases}
\end{equation}
Furthermore, let $d:\bar{\Omega}\times \RR_+\to \RR$ be the signed distance to $M_t$, namely
\begin{equation}\label{equ:defofd}
d(x, t)= \begin{cases}\operatorname{dist}(x, M_t) & \text { if } x \in K_t, \\ -\operatorname{dist}(x, M_t) & \text { if } x \notin K_t .\end{cases}
\end{equation}
Now, remembering \eqref{sing_cont_in_pn} we choose an open set $U\subset\RR^{n+1}$, such that
\begin{equation}
S_{T,\eps}\subset U\Subset P\cup N.
\end{equation}
Lemma \ref{lem:fbcut} gives us a smooth function $\phi: \mathbb{R}^{n+1} \rightarrow[0,1]$, such that
\begin{equation}
\phi|_{\bar{U}} \equiv 1,\quad \spt\phi\subset P\cup N,\quad \textrm{and}\quad D \phi\cdot \nu_{\partial \Omega}=0 \textrm{ on } \partial \Omega. 
\end{equation}
Define a function $w:\left(X \cup (\phi^{-1}(0)\cap\bar{\Omega})\right) \times[T,T+ \eps) \rightarrow \mathbb{R}$ by
\begin{equation}
w(x, t)=\phi(x) f(x, t)+(1-\phi(x)) d(x, t).
\end{equation}

\bigskip

Note that on $\phi^{-1}((0,1))$ and for every $t\in[T,T+\eps)$ the zero sets (respectively negative/positive sets) of $w$, $d$ and $f$ coincide. In particular, we have
\begin{equation}
w(x,t)=0 \,\, \Leftrightarrow \,\, x\in M_t.
\end{equation}
Also recall that the outer free boundary flow $\{M_t\subset\bar{\Omega}\}$ is a subsolution of the free boundary flow, as has been shown in \cite[Proposition 5.2]{Bao}. Moreover, observe that $Dd=\nu_{M_t}$ on the smooth part of $M_t$, and $|Du|=1/|H|$ on the smooth part of $M_t\cap (P\cup N)$. This yields
\begin{equation}
Dw|_{(M_t\setminus S_t)\times\{t\}}=\left((1-\phi)+\frac{\phi}{|H|}\right)\nu_{M_t}.
\end{equation}

Now, let $Z$ be an $\eps$-neighborhood of $M_T \setminus U$, where $\eps$ is small enough, such that $\bar{Z}$ is disjoint from $S_{T,\eps}$ and $w(\cdot, T)$ is smooth with nonvanishing gradient on $\bar{Z}$. 
In fact, choosing $\tau \in (0, \eps)$ sufficiently small, $w$ is smooth with nonvanishing gradient on $\bar{Z}\times [T,T+\tau]$, 
and we have
\begin{equation}
\bigcup_{t \in[T,T+ \tau]} M_t \setminus U \subseteq Z.
\end{equation}
Furthermore, let $Y\subseteq U$ be an open set, such that
\begin{equation}
Y\cap \bar{\Omega}=\bigcup_{t \in(T-\eps,T+ \eps)} M_t \cap U.
\end{equation}
Since $Y\subseteq U$, we have $w\equiv f$ on $Y$, and thus considering (\ref{equ:defoff}) we see that $\{ \pm w(\cdot,t) \geq c\}_{t\in [T,T+\tau]}\cap Y$ are free boundary subsolutions in $Y$.

Finally, since $D\phi \cdot \nu_{\partial \Omega}=0$ on $\partial \Omega$ by our choice of cutoff function, on $Z\cap \partial \Omega$ we get
\begin{equation}
Dw\cdot \nu_{\partial \Omega}=(1-\phi)Dd\cdot \nu_{\partial \Omega},
\end{equation}
where we also took into account that $Df\cdot \nu_{\partial \Omega}=0$ by the free boundary condition on the smooth part.
Moreover, by the convexity of $\Omega$, we have $\pm D d \cdot \nu_{\partial \Omega}\ge 0$ when $\pm d\geq 0$. This shows that $\pm Dw\cdot \nu_{\partial \Omega}\geq 0$ when $\pm w\geq 0$.
We can thus apply Theorem \ref{thm:hwcpam15} to conclude that $\{M_t\}_{t\in [T,T+\tau]}$ is the free boundary level set flow in $\bar{\Omega}$ starting from $M_T$. As the exact same argument would have applied for $M^{\prime}_t$, this proves all our three flows agree until time $T+\tau$, and thus that $T_{\text {disc }}\geq T+\tau$.
\end{proof}

\bigskip

\bibliography{BaoHaslhofer_fb_cyl}

\bibliographystyle{abbrv}

\end{document}